\documentclass[12pt]{article}
\usepackage{epsfig,amsmath,amsfonts,amssymb,latexsym,color,amsthm,rotating,multirow,hhline}
\newtheorem{theorem}{Theorem}[section]
\newtheorem{prop}[theorem]{Proposition}
\newtheorem{lemma}[theorem]{Lemma}
\newtheorem{corr}[theorem]{Corollary}

\newcommand{\E}{{\mathbb E}}
\newcommand{\RR}{{\mathbb R}}
\newcommand{\R}{{\mathbb R}}
\newcommand{\Z}{{\mathbb Z}}
\newcommand{\PP}{\mathbb P}
\newcommand{\poi}{\mathcal{P}}
\newcommand{\mat}{\mathcal{M}}

\newcommand{\df}{\textbf}

\newcounter{mycount}
\newenvironment{romlist}{\begin{list}{\rm(\roman{mycount})}%
   {\usecounter{mycount}\labelwidth=1cm\itemsep -0pt}}{\end{list}}

\begin{document}

\title{Stable Poisson Graphs in One Dimension}

\author{Maria Deijfen
\thanks{Department of Mathematics, Stockholm University, 106 91 Stockholm.
{\tt mia at math.su.se}}
 \and Alexander E. Holroyd
 \thanks{Microsoft Research, 1 Microsoft Way, Redmond, WA 98052, USA;
{\tt holroyd at microsoft.com, peres at microsoft.com}}
\and Yuval Peres $^\dag$}

\date{April 2011}

\maketitle

\thispagestyle{empty}

\begin{abstract}
Let each point of a homogeneous Poisson process on $\RR$
independently be equipped with a random number of stubs
(half-edges) according to a given probability distribution
$\mu$ on the positive integers.  We consider schemes based on
Gale-Shapley stable marriage for perfectly matching the stubs
to obtain a simple graph with degree distribution $\mu$.  We
prove results on the existence of an infinite component and on
the length of the edges, with focus on the case $\mu(\{2\})=1$.
In this case, for the random direction stable matching scheme
introduced by Deijfen and Meester we prove that there is no
infinite component, while for the stable matching of Deijfen,
H\"aggstr\"om and Holroyd we prove that existence of an
infinite component follows from a certain statement involving a
{\em finite} interval, which is overwhelmingly supported by
simulation evidence.

\end{abstract}

\renewcommand{\thefootnote}{}

\section{Introduction}\label{intro}

Let $\poi$ be a homogeneous Poisson process with intensity 1 on
$\RR^d$ and $\mu$ a probability measure on the strictly
positive integers. We shall study translation-invariant simple
random graphs whose vertices are the points of $\poi$ and
where, conditional on $\poi$, the degrees of the vertices are
i.i.d.\ with law $\mu$. Previously, Deijfen \cite{D} has studied
achievable moment properties for the edges, and Deijfen,
H\"aggstr\"om and  Holroyd  \cite{DHH} have studied the question
of whether the graph contains a component with infinitely many
vertices. In the latter work a particular matching scheme,
called the stable multi-matching, was introduced, leading to a
number of challenging open questions. Here we restrict to $d=1$
and the focus is on the case $\mu(\{2\})=1$, one of the
simplest cases for which the question of existence of an
infinite component is non-trivial. For the stable
multi-matching and a variant of it with prescribed random stub
directions, we prove results on the component structure and on
the length of the edges. Figure 1 shows schematic pictures of the
two matchings, which are described below.

\begin{figure}
\centering
\resizebox{.9\textwidth}{!}{\includegraphics{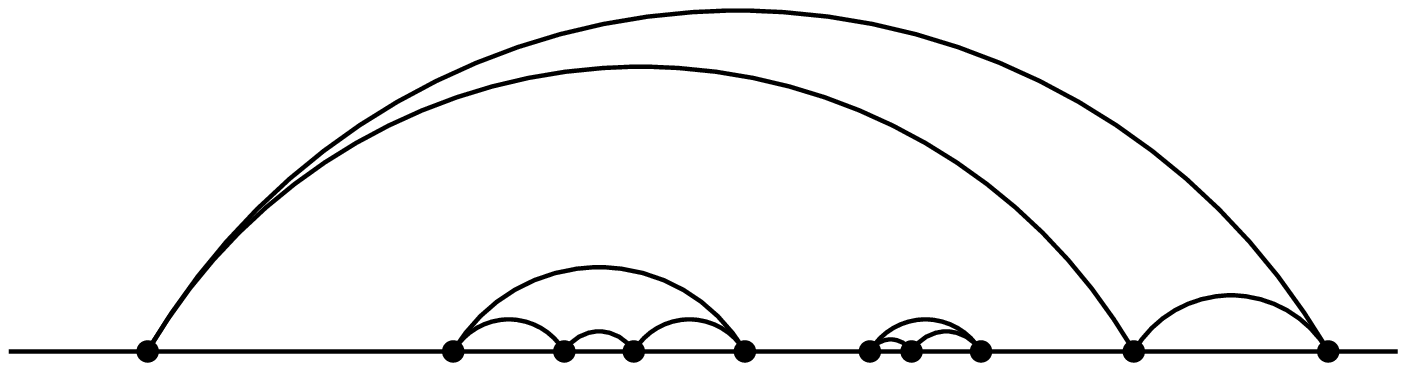}}
\resizebox{.9\textwidth}{!}{\includegraphics{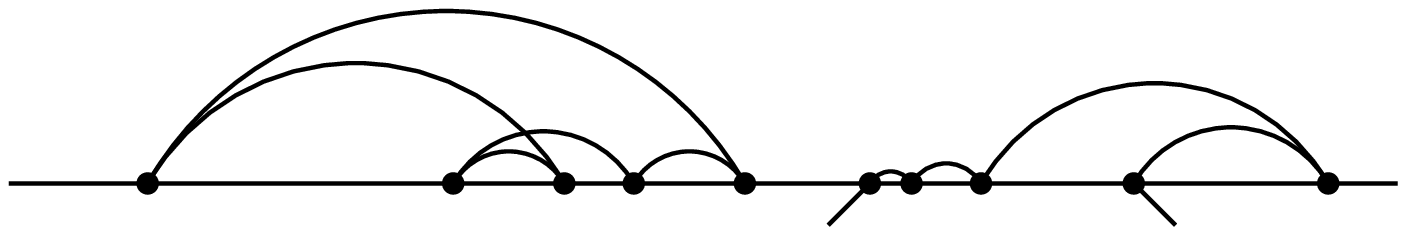}}
\caption{The stable multi-matching (top) and the random direction stable multi-matching,
for 10 vertices on a finite interval, with 2 stubs per vertex.
} \label{example}
\end{figure}

First we formally describe the objects that we will work with.
Write $[\poi]:=\{x\in\RR:\poi(\{x\})>0\}$ for the support, or
point-set, of $\poi$. Let $\xi$ be a random integer-valued
measure on $\RR$ with the same support as $\poi$, and which,
conditional on $\poi$, assigns i.i.d.\ values with law $\mu$ to
the elements of $[\poi]$. The pair $(\poi,\xi)$ is a marked
point process with positive integer-valued marks. For
$x\in[\poi]$ we write $D_x$ for $\xi(\{x\})$ and  interpret
this as the number of stubs at vertex $x$.

A {\bf matching scheme} for a marked process $(\poi,\xi)$ is a point process
$\mat$ on the space of unordered pairs of points in $\RR$, with the property
that almost surely for every pair $(x,y)\in[\mat]$ we have $x,y\in[\poi]$,
and such that in the graph $G=G(\poi,\mat)$ with vertex set $[\poi]$ and edge
set $[\mat]$, each vertex $x$ has degree $D_x$.  Our primary interest is in
the connected components of $G$.  The matching schemes under consideration
will always be {\bf simple}, meaning that $G$ has almost surely no self-loops
and no multiple edges, and {\bf translation-invariant}, meaning that $\mat$
is invariant in law under the action of all translations of $\RR$. We say
that a translation-invariant matching is a {\bf factor} if it is a
deterministic function of the Poisson process $\poi$ and the mark process
$\xi$, that is, if it does not involve any additional randomness. We write
$\PP$ and $\E$ for probability and expectation on the probability space
supporting the random triplet $(\poi,\xi,\mat)$.

Let $(\poi^*,\xi^*,\mat^*)$ be the Palm versions of $(\poi,\xi,\mat)$ with
respect to $\poi$ and write $\PP^*$ and $\E^*$ for the associated probability
law and expectation operator. Informally speaking, $\PP^*$ describes the
conditional law of $(\poi,\xi,\mat)$ given that there is a point at the
origin, with the mark process and the matching scheme taken as stationary
background; see e.g.\ \cite[Chapter 11]{K} for more details.  Since $\poi$ is
a Possion process, we have $[\poi^*]\stackrel{d}{=} [\poi]\cup \{0\}$.

We now define the two matching schemes that will be analyzed in
the paper.

\subsubsection*{Stable multi-matching}

The concept of stable matching was introduced by Gale and
Shapley \cite{GS}.  It has been studied in \cite{HPPS} and
\cite{HP} in the context of spatial point processes (with
$\mu(\{1\})=1$ in our notation). A natural generalization to
other degree distributions $\mu$ was introduced in \cite{DHH}
and is referred to as the stable multi-matching. Formally, a
matching scheme $\mat$ is said to be a {\bf stable
multi-matching} if a.s., for any two distinct points $x,y \in
[\poi]$, either they are linked by an edge or at least one of
them has no incident edges longer than $|x-y|$.  Here and
throughout, distance and edge length refer to the Euclidean
norm $|\cdot|$ on $\RR$.

We will restrict our attention to the case when $\poi$ is a Poisson process.
For this case, it was proved in \cite[Proposition 2.2]{DHH} that there is an
a.s.\ unique stable multi-matching, which moreover can be constructed by the
following iterative procedure.  First connect all mutually closest pairs of
points in $[\poi]$ and remove one stub from each of these point. Then call
two points compatible if they do not already have an edge between them and if
both of them have at least one stub left. Connect all mutually closest
compatible pairs and remove one stub from each of the points just matched. Repeat indefinitely.
See \cite[Propostion 2.2]{DHH}.\enlargethispage*{1cm}

\subsubsection*{Random direction stable multi-matching}

We introduce a variant of stable multi-matching where the directions of the
edges are prescribed independently of the Poisson process. As described
above, the process $\xi$ assigns a mark $D_x$ to each point $x\in[\poi]$. Let
$\psi$ be a second mark process which, conditionally on $\poi$ and $\xi$,
assigns an integer $R_x\sim\mbox{Binomial}(D_x,1/2)$ independently to each
point $x\in[\poi]$.  We think of $R_x$ as the number of stubs incident with
$x$ that are to be matched to the right of $x$.  If $x<y$, and $(x,y)$ is an
edge of a matching scheme $\mat$, we call $(x,y)$ a right-edge of $x$, and a
left-edge of $y$. A matching scheme $\mat$ is now said to be a {\bf random
direction stable multi-matching} if each point $x\in[\poi]$ has exactly $R_x$
incident right-edges and if a.s., for any two distinct points $x,y \in
[\poi]$ with $x<y$, either they are linked by an edge, or $x$ has no incident
right-edges longer than $|x-y|$, or $y$ has no incident left-edges longer
than $|x-y|$.

Let each point $x\in[\poi]$ be equipped with $R_x$ stubs
pointed to the right and $L_x:=D_x-R_x$ stubs pointed to the
left, and consider the following iterative procedure for
matching right-stubs to left-stubs.  First consider all pairs
of consecutive points in $[\poi]$. Create an edge between every
such pair $x<y$ such that $x$ has at least one right-stub and
$y$ has at least one left-stub, and remove the corresponding
stubs. Then consider pairs of points in $[\poi]$ with precisely
one point in $[\poi]$ in between them. Create an edge between
every such pair of points $x<y$ such that $x$ has at least one
right-stub and $y$ has at least one left-stub left, and remove
the corresponding stubs. Continue indefinitely, with pairs of
points separated by an increasing number of points. This
procedure has previously been studied in \cite{DM}. We show in
Section \ref{prel} that it leads to the unique stable
multi-matching subject to the prescribed (random) directions for the
edges.

\subsection{Results}\label{results}

In this section we collect the main results. The proofs are
then given in Section \ref{proofs}. The first result concerns
uniqueness of the infinite component.

\begin{prop}\label{prop:unik}
For a Poisson process on $\RR$ and any degree distribution, in
the stable multi-matching and the random direction stable
multi-matching, there is at most one infinite component.
\end{prop}

The next result asserts that, in the case $\mu(\{2\})=1$ of two
stubs per vertex, the random direction stable multi-matching
has no infinite components. For other degree distributions the
existence of an infinite component remains an open question.
Part (b) of the theorem however provides some information on
the edge length. See \cite[Theorem 4.1]{DM} and \cite[Theorem
2]{HPPS} for related results.

For $x\in[\poi]$, let $X_x$ denote the average length of all
edges incident to $x$, and write $X=X_0$ for the value at the
origin in the Palm version of the process.

\begin{theorem}\label{th:smrd}For a Poisson process on $\RR$,
consider the random direction stable multi-matching.
\begin{romlist}
\item For $\mu(\{2\})=1$, almost surely there is no
    infinite component.
\item For any degree distribution with bounded support,
    we have \linebreak $\E^*[X^{1/2}]=\infty$.
\end{romlist}
\end{theorem}

Turning to the stable multi-matching, it was proved in \cite[Theorem
1.2(b)]{DHH} that there is no infinite component when the only possible
values for the degrees are 1 and 2, with a strictly positive probability of
degree 1.  In $d\geq 2$ it was also proved that there is an integer $k=k(d)$
such that if all vertices almost surely have degree at least $k$, then there
is almost surely an infinite component, \cite[Theorem~1.2(a)]{DHH}.  Note
that, by ergodicity, the event that there exists an infinite component has
probability 0 or 1 for any degree distribution. The following result relates
the existence of an infinite component for the case $\mu(\{1,2\})=1$ in $d=1$
to a certain property concerned with the lengths of the edges. Let $M_x$
denote the length of the longest edge incident to $x\in[\poi]$, say that $x$
{\bf desires} a site $y\in\mathbb{R}$ if $|y-x|<M_x$ and write $N$ for the
number of points in $[\poi]$ that desire the origin.

\begin{theorem}\label{th:stable}\sloppypar For a Poisson process on $\RR$,
consider the stable multi-matching.
\begin{romlist}
\item For any degree distribution, if there is no infinite
    component, then \linebreak $N=\infty$ almost surely.
\item If $\mu(\{1,2\})=1$ and there is an infinite
    component, then $N<\infty$ almost surely.
\end{romlist}
\end{theorem}

\noindent For degree distributions with $\mu(\{1,2\})=1$,
existence of an infinite component in the stable multi-matching
is hence equivalent to $N<\infty$. On the other hand, $N$ is
related to edge lengths, as follows. Write $M=M_0$.

\begin{lemma}\label{le:NM}
For any translation-invariant matching scheme, we have that
$\E^*[M]<\infty$ if and only if $\E[N]<\infty$.
\end{lemma}

In view of this relation, $E^*[M]<\infty$ would imply that
$N<\infty$, and thereby establish the existence of an infinite
component for $\mu(\{2\})=1$ in the stable multi-matching.
However, the best result we have in this direction is the
following, which applies in any dimension $d\geq 1$ (the stable
multi-matching is defined analogously in all dimensions; see
\cite{DHH}).

\begin{prop}\label{prop:M_bound}
For a Poisson process of intensity $1$ on $\RR^d$, and any
degree distribution with bounded support, in the stable
multi-matching we have $\PP^*(M>t)\leq ct^{-d}$ for some
$c\in(0,\infty)$ (depending only on $d$ and the bound on
degree).
\end{prop}

\paragraph{A ``statistical proof'' of percolation.}
It is not rigorously known whether the stable multi-matching with $\mu(\{2\})=1$
will have an infinite component in $d=1$.  However, in Section~\ref{MC-proof} we
present compelling evidence that this is indeed the case. Specifically, we
will define a certain event $G_L$ {\em in terms of a Poisson process on the
bounded interval} $[0,L]$. We will prove rigorously that for any $L>0$,
\[
\PP(G_L)>0.968 \quad\text{implies existence of an infinite component.}
\]
On the other hand, since $G_L$ is defined in terms of a bounded
interval, its probability can be estimated by Monte-Carlo
simulation.  Such simulations provide overwhelming statistical
evidence that
$$\PP(G_{13000})>0.968.$$

The random direction stable multi-matching and the stable multi-matching are
hence qualitatively different: when the directions of the stubs are
prescribed randomly, there is no infinite component, while when the
directions are prescribed by the positions of the Poisson points (as in the
stable multi-matching) there is an infinite component. Figure 2 shows
simulation pictures of the random direction stable multi-matching and the
stable multi-matching, respectively, with $\mu(\{2\})=1$ in $d=1$.

\begin{figure}
\centering
\resizebox{.9\textwidth}{!}{\includegraphics{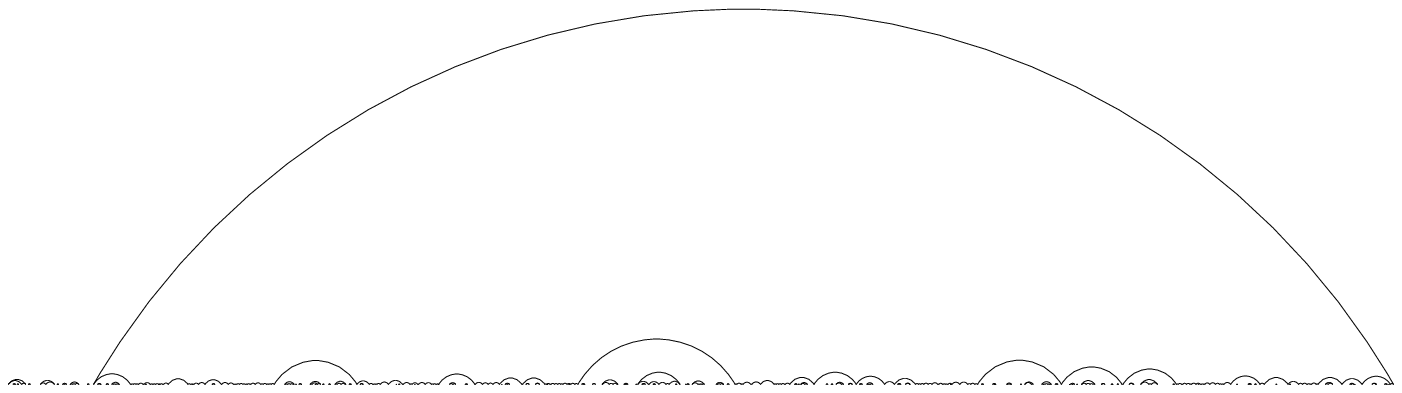}}
\resizebox{.9\textwidth}{!}{\includegraphics{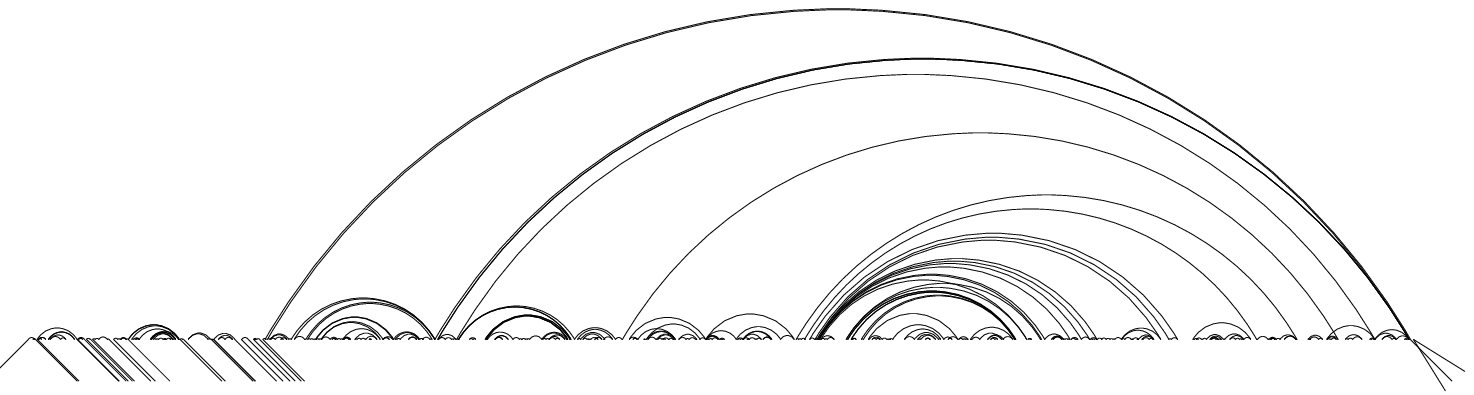}}
\caption{The stable multi-matching (top) and the random direction stable multi-matching,
with 2 stubs per vertex, for 500 uniformly random points on an interval. }
\end{figure}

Say that an edge $(x,y)\in[\mat]$ {\bf crosses} a site $z\in\RR$ if $x<z<y$.
Our last result is the following.

\begin{prop}\label{prop:odd_deg}
For a Poisson process on $\RR$, and any degree distribution
$\mu$ whose support includes some odd integer, for any factor
matching scheme, the number of edges that cross the origin is
infinite.
\end{prop}

\noindent If the number of edges that cross the origin if
infinite, then clearly also $N=\infty$. Hence, appealing to
Lemma \ref{le:NM}, Proposition \ref{prop:odd_deg} implies that
$\E^*[M]=\infty$ in any factor matching scheme for degree
distributions whose support contains an odd integer. Since the
stable multi-matching is a factor, combining Proposition
\ref{prop:odd_deg} with Theorem \ref{th:stable}(b) gives an
alternative proof (in $d=1$) of the result of \cite{DHH} that
the stable multi-matching does not percolate when the only
possible values for the degrees are 1 and 2 and the probability
of degree 1 is strictly positive. For degree distributions with
support on larger values this approach is inconclusive, since
Theorem \ref{th:stable}(b) does not apply.

The rest of the paper is organized as follows. In Section~\ref{prel}, a few
preliminary results are collected. The above results are then proved in
Section~\ref{proofs}.  In Section~\ref{MC-proof}, existence of a infinite
component in the degree 2 case in $d=1$ is shown to follow from the assertion
that a certain finite event has large enough probability, which is
convincingly supported by Monte-Carlo simulation. Section~\ref{further}
contains examples demonstrating that there is no general relation between the
edge length and the existence of an infinite component valid for any matching
scheme, and that for point processes other than the Poisson process, both
percolation and non-percolation are possible for the stable multi-matching in
the degree $2$ case. Finally, in Section~\ref{open} some directions for
further work are presented. For background on the problem we refer to
\cite[Section 2.1]{DHH} and \cite[Section 1]{D}.

\section{Preliminaries}\label{prel}

We first show that the iterative procedure described for the
stable multi-matching with random directions leads to the
unique stable multi-matching with the prescribed directions for
the edges.

\begin{prop}
Let $(\poi,\xi,\psi)$ be a doubly marked Poisson process.
Almost surely, the iterative multi-matching procedure described
in the introduction exhausts the set of stubs, and the limiting
graph (after an infinite number of iterations) is a random
direction stable multi-matching.  No other such matching scheme
exists.
\end{prop}

\begin{proof}
Let $\poi'_r$ (respectively $\poi'_l$) be the process of points
with at least one unmatched right-stub (left-stub) on them
after the above matching procedure is completed. By symmetry
$\poi'_r$ and $\poi'_l$ have the same intensity and they are
both ergodic point processes. Hence either both have a.s.\
infinitely many points or both have a.s.\ no points. The first
option however would produce a contradiction, since the
iterative procedure could then be applied to the remaining
configuration of stubs giving rise to edges that would have
been created already in the original procedure.

That the resulting multi-matching is stable subject to the
prescribed (random) directions follows from the definition: an unstable pair of
points -- that is, a pair $x$ and $y$ with $x<y$ with no edge
between them and where $x$ ($y$) has an edge connected to the
right (left) of $y$ ($x$) -- would have had an edge created
between them at some stage of the matching procedure. That it
is the unique matching with this property follows by induction
over the stages in the algorithm to show that each edge that is
present in the resulting configuration must be present in any
stable matching of the stubs with the prescribed directions.
\end{proof}

We proceed by formulating a version of the mass transport
principle suitable for our needs. For background, see
\cite{BLPS}. A {\bf mass transport} is a random measure $T$ on
$(\RR^d)^2$ that is invariant in law under translations of
$\RR^d$. For Borel sets $A,B\subset\RR^d$, we interpret
$T(A,B)$ as the amount of mass transported from $A$ to $B$.
Write $Q$ for the unit cube $[0,1)^d$.

\begin{lemma}[Mass Transport Principle]\label{le:mt}
Let $T$ be a mass transport. Then
\[
\E\, T(Q,\RR^d) =\E\, T(\RR^d,Q) \, .
\]
\end{lemma}

\begin{proof}
\[
\E\, T(Q,\RR^d)  =  \sum_{z\in\Z^d}\E\, T(Q,Q+z)
 =   \sum_{z\in\Z^d}\E\, T(Q-z,Q) = \E\, T(\RR^d,Q) \, .
\qedhere
\]
\end{proof}

Lemma \ref{le:NM} is now easily established using the mass transport principle.

\begin{proof}[Proof of Lemma \ref{le:NM}]
Consider the mass transport in which each point $x\in[\poi]$ sends out mass
$2M_x$, and distributes it uniformly to the interval $(x-M_x,x+M_x)$. The
expected mass sent out from the unit interval $[0,1)$ equals $2 \,\E^*[M]$.  On the other hand, the mass received by the unit interval is $\int_0^1 N_x \,dx$, where $N_x$ denotes the number of points that desire $x\in\RR$.  Hence the expected mass received by the unit interval is $\E N$.  The result hence follows from the mass
transport principle.
\end{proof}

Next we observe that an infinite component in a
translation-invariant matching scheme must be unbounded both to
the right and to the left, that is, for any $r\in\RR^+$ it must
contain points both to the right of $r$ and to the left of
$-r$.

\begin{lemma}\label{le:both_directions}
A translation-invariant matching scheme almost surely cannot
have an infinite component that is unbounded in only one
direction.
\end{lemma}

\begin{proof}
Assume that there is a matching scheme that with positive
probability gives rise to an infinite component that is
unbounded in only one direction, say to the left, and consider
the mass transport in which each vertex in such an infinite
component sends mass 1 to the rightmost point in the component.
With positive probability such a rightmost point is located in
the unit interval, which then receives infinite mass. The
expected mass sent out from the unit interval is however
bounded by 1, so we have a contradiction with the mass
transport principle.
\end{proof}

Finally, the following result will be of use in proving Theorem
\ref{th:stable}(b).

\begin{lemma}\label{lemma:next_one}
Let $\Gamma$ be a translation-invariant simple point process of
finite intensity on $\RR$.  For $x\in[\Gamma]$, write $Z_x$ for
the maximum of the distances from $x$ to the nearest point of
$[\Gamma]$ on the left and the nearest point on the right.  The
number of points $x\in[\Gamma]$ with $Z_x>|x|$ is finite almost
surely.
\end{lemma}

\begin{proof}
Consider the mass transport in which an interval sends out mass equal to its
length, and the mass sent out by an interval $(x,y)$ between consecutive
points $x<y$ of $[\Gamma]$ is distributed uniformly to the interval
$(x-(y-x),y+(y-x))$. If there were infinitely many points $x\in[\Gamma]$ with
$Z_x>|x|$, then the unit interval would receive infinite mass, which
conflicts with the mass transport principle, since the mass sent out from the
unit interval equals 1.
\end{proof}

\section{Proofs}\label{proofs}

We now proceed to prove the results in Section \ref{results},
starting with the uniqueness of an infinite component. We say
that two edges $(a,b)$ and $(c,d)$ in $[\mat]$ {\bf cross} each
other if $a<c<b<d$.

\begin{proof}[Proof of Proposition \ref{prop:unik}]
Observe that both matching schemes have the property that,
\begin{equation}\label{eq:cross_conn}
\begin{array}{l}
\mbox{if two edges $(a,b)$ and $(c,d)$ cross each other, then}\\
\mbox{the edge $(c,b)$ must also be present in the matching.}
\end{array}
\end{equation}
-- this follows from the definitions of the matching schemes. Two distinct
components hence cannot have crossing edges. However, by Lemma
\ref{le:both_directions}, any infinite component must be
unbounded in both directions.  Hence two distinct infinite
components would necessarily have crossing edges.
\end{proof}

\begin{proof}[Proof of Theorem \ref{th:smrd} (i)]
Recall that $L_x$ (respectively $R_x$) is the number of edges incident with
$x\in[\poi]$ that are connected to the left (right) of $x$. Let
$\mu(\{2\})=1$.  We call $x$ a {\bf bird} if $L_x=R_x=1$, a {\bf left-beak}
if $L_x=2$ and a {\bf right-beak} if $R_x=2$ (see Figure 1). Let
$(\cdots<)x_1<\cdots<x_k(<\cdots)\in[\poi]$ be the ordered vertices of some
(finite or infinite) component of the stable multi-matching (recall by
Lemma~\ref{le:both_directions} that a component is either finite or unbounded
in both directions). Clearly, if the component is finite, then its leftmost
point is a right-beak and its rightmost point is a left-beak. We claim that
if $x_i$ is a right-beak, and not one of these extreme points of the
component, then $x_{i+1}$ is a left-beak. To check this, let the two edges
from $x_i$ have their other endpoints at $x_j$ and $x_k$, where
$x_i<x_j<x_k$.  We claim that the other neighbour of $x_j$ must lie left of
$x_i$.  To see this, follow the path formed by the cluster starting with the
edge $(x_i,x_j)$ - eventually we must leave the interval $[x_i,x_k]$, since
the cluster contains points to the left of $x_i$.  When we do so, it is via
an edge that crosses $(x_i,x_k)$.  Unless it is the first edge encountered
after $x_j$, this entails a violation of (\ref{eq:cross_conn}).  Thus $x_j$
is a left beak.  Now there cannot be any further vertices of the cluster in
the interval $[x_i,x_j]$, since such a vertex would have an incident edge
crossing $(x_i,x_j)$, again contradicting (\ref{eq:cross_conn}).

Therefore, the non-extreme vertices of a component consist of birds together
with consecutive right-beak/left-beak pairs. Note that between the points of
a component there may be points belonging to other components, but since two
components cannot have crossing edges, any other such component must lie in a
single interval $(x_i,x_{i+1})$.

Now consider the function $F:\RR\to\Z$ defined by $F(0)=0$, and
$$F(y)-F(x)=\sum_{t\in [\poi]\cap [x,y)} (L_t-R_t), \quad x<y.$$
Thus, $F$ takes a up-step (of size 2) at each left-beak and a down-step at
each right-beak. Hence it is a continuous-time simple symmetric random walk
on the even integers.  On the other hand, by the observations above
concerning components, if there is an infinite component, then $F$ is bounded
above a.s.\ by some (random) constant, which is impossible.
\end{proof}

\begin{proof}[Proof of Theorem \ref{th:smrd} (ii)]
We use a variant of an argument from \cite[Theorem~2(b)]{HPPS}. For $A\subset
\RR$, write $R(A)$ for the total number of right-stubs at points
$x\in[\poi]\cap A$, that is, $R(A)=\sum_{x\in[\poi]\cap A}R_x$, and define
$L(A)$ analogously as the total number of left-stubs in $A$. Furthermore, for
$A,B\subset \RR$, let $R(A\to B)$ denote the number of right-stubs in $A$
that are matched with left-stubs in $B$, and $D(A\leftrightarrow B)$ the
total number of edges connecting points in $[\poi]\cap A$ to points in
$[\poi]\cap B$. Write $k$ for the supremum of the support of $\mu$. For
$t>0$, we have
\begin{eqnarray*}
\E R\big([0,2t]\to[0,2t]^c\big) & = & \frac{1}{2}\,\E D\big([0,2t]\leftrightarrow[0,2t]^c\big)\\
& \leq & \frac{k}{2}\,\int_0^{2t}\PP^*\big(X>x\wedge (2t-x)\big)\;dx\\
& = & k\cdot \E^*[X\wedge t].
\end{eqnarray*}
Furthermore, since $\mu$ has bounded support, we can use the central limit theorem
to get that
$$
\E R\big([0,2t]\to[0,2t]^c\big) \geq \E\Big[\big(R[0,2t]-L[0,2t]\big)^+ \Big]\sim ct^{1/2}
$$
as $t\to\infty$ for some $c>0$. Hence $t^{-1/2}\E^*[X\wedge
t]\geq c'$ for sufficiently large $t$ and some $c'>0$. On the
other hand, if $\E^*[X^{1/2}]<\infty$, then
$t^{-1/2}\E^*[X\wedge t]\to 0$ as $t\to \infty$ by the
dominated convergence theorem, a contradiction.
\end{proof}

\begin{proof}[Proof of Theorem \ref{th:stable} (i)] Let
$$
H=\{x\in[\poi]: M_x>|x|-1\},
$$
that is, $H$ is the set of vertices that desire some point in the unit
interval $(-1,1)$.  Write $\widetilde{N}$ for the cardinality of $H$.  We
will show that $\widetilde{N}=\infty$ a.s.  By symmetry this implies that
with positive probability infinitely many vertices in $(1,\infty)$ desire
$1$.  However, on the latter event, for any $a>1$, infinitely many vertices
in $(a,\infty)$ desire $a$, so by ergodicity it follows that $N=\infty$ a.s.

First we show that $\PP(\widetilde{N}=0)=0$. Assume for contradiction that
$\PP(\widetilde{N}=0)>0$.  For a configuration $(\poi,\xi)$ with
$\tilde{N}=0$, consider a modified configuration where a vertex is added
uniformly at random in $[0,1]$ independently of $\poi$. If follows from
\cite[Lemma 18]{HPPS} and a straightforward modification of \cite[Lemma
16]{HPPS} that all stubs at this vertex would be unmatched in the stable
multi-matching, which contradicts \cite[Proposition 2.2]{DHH}.

Now assume that all components are finite a.s., and suppose that for a
contradiction that $\PP(\widetilde{N}<\infty)>0$. For a configuration
$(\poi,\xi)$ with $\widetilde{N}<\infty$, consider a modified configuration
where the vertices in $H$ are removed, along with all vertices in their
components. The number of vertices that are removed is almost surely finite.
But in this configuration, we have $\widetilde{N}=0$, which is a
contradiction to a straightforward modification of \cite[Lemma~18]{HPPS}.
\end{proof}

\begin{proof}[Proof of Theorem \ref{th:stable} (ii)]
First note that, when the only possible values for the degrees
are 1 and 2, the stable multi-matching cannot contain any
crossing edges.  If $a<c<b<d$ and the edges $(a,b)$ and $(c,d)$
are present, then, as pointed out in the proof of
Proposition~\ref{prop:unik}, the edge $(c,b)$ must also be
present in the matching.  But if $b-a>d-b$, then $b$ and $d$
desire each other, and are hence connected by an edge, so $b$
has degree at least $3$.  Similarly, if $b-a<d-b$ then $c$ has
degree at least $3$.

Lemma \ref{le:both_directions} and the fact that edges do not cross imply
that an infinite component must consist of a set of degree-2 vertices,
unbounded in both directions, with an edge between each consecutive pair.  It
follows from Lemma~\ref{lemma:next_one} that the number of vertices in this
infinite component that desire the origin is finite almost surely.

As for the finite components, each must be contained in a
single interval defined by an edge of the infinite component
(since there are no crossing edges).  Note also (although this
observation will not be needed) that a component of size $k$
must consist of vertices $x_1<\cdots<x_k$ with edges
$(x_i,x_{i+1})$ for all $i=1,\ldots, k-1$ together with the
edge $(x_1,x_k)$.

Now let $I_0$ denote the interval defined by the edge in the
infinite path that crosses the origin. This interval is finite
and hence contains almost surely finitely many points of
$[\poi]$ in finite components. These points might desire the
origin. A vertex $x>0$ (respectively, $x<0$) in a finite
component outside this interval however cannot desire the
origin: if it did, it would also desire the left-most
(right-most) end-point of the interval $I_x$ defined
analogously to $I_0$. But this vertex also desires $x$, which
means that there would be an edge between them.

We conclude that $N<\infty$ almost surely, as desired.
\end{proof}

\begin{proof}[Proof of Proposition \ref{prop:M_bound}]
Say that a point $x\in[\poi]$ is {\bf $t$-bad} if $M_x>t$. If
$D\leq k$ almost surely, then there can be at most $k$ $t$-bad
points in the ball $B(0,t/2)$. Hence
$$
k\geq \E[\mbox{number of $t$-bad points in }B(0,t/2)\}]= \mbox{vol}(B(0,t/2))\,\PP^*(M>t),
$$
giving the result.
\end{proof}

\begin{proof}[Proof of Proposition \ref{prop:odd_deg}]
Assume that the number of edges that cross the origin is finite with positive
probability. On the event that the origin is crossed by finitely many edges,
the same is true for any other $x\in\RR$, since the difference between the
number of edges crossing $x$ and the number of edges crossing the origin is
bounded above by the total degree of the vertices between the origin and $x$.
Now consider the intervals between the points $x\in[\poi]$ with odd degrees.
When passing a point with odd degree, the number of crossing edges changes
parity, that is, if it is even (odd) immediately to the left of the point, it
is odd (even) to the right. When passing a point with even degree on the
other hand, the parity of the number of crossing edges remains unchanged.
This means that we can assign the value 0 (even number of edges crossing) or
1 (odd number of edges crossing) to the intervals separating the odd degree
vertices in a deterministic way (indeed, the stable multi-matching is a
factor). Furthermore, the odd degree vertices constitute a Poisson process.
Now, \cite[Lemma 11]{HPPS} asserts that it is impossible to assign
alternating values 0 and 1 to the intervals separating the points of a
Poisson process as a factor of the Poisson process. Here we need the stronger
statement that this cannot be done even using the randomness in the degrees
of the vertices and in the position of the even degree vertices. This however
follows from a straightforward modification of the proof of \cite[Lemma
11]{HPPS}.
\end{proof}

\section{Percolation for the stable multi-matching with $D\equiv 2$}\label{MC-proof}


If the stable multi-matching almost surely has an infinite component, then
there is a strictly positive probability $p$ that a given vertex belongs to
this component. Simulations of the stable multi-matching with $D\equiv 2$ on
large finite cycles indicate a largest component comprising about $0.3$ of
the vertices (see the top row of Table~\ref{sim} in Section~\ref{open}
below). This suggests the existence of an infinite component with $p\approx
0.3$. In this section we show that percolation indeed follows from the
assumption that a certain finite event has sufficiently large probability.
Furthermore, we give overwhelming statistical evidence for this assumption.

The key concept for the proof is the \df{core (stable) multi-matching}, which
we define next (in the more general setting of arbitrary dimension and
numbers of stubs). Let $S\subset\R^d$ be a bounded set, let $P\subset S$ be a
finite set of points, and let $(D_x)_{x\in P}$ be positive integers
representing numbers of stubs.  Let $\widetilde P=P\cup\{S^C\}$, where
$S^C:=\R^d\setminus S$.  (We will treat $S^C$ like an additional point; it
will not form part of the matching, but will affect the notion of closest
points.) Assume that all distances between pairs of elements of $\widetilde
P$ are distinct. Assign $D_x$ stubs to each point $x\in P$, and one stub to
$S^C$.  Repeat the following operations. From each point $x\in P$ that
currently has an unused stub, assign an arrow pointing to the closest other
element of $\widetilde P$ among those that have at least one unused stub and
do not already have an edge to $x$. Then, for every pair $x,y\in P$ whose
arrows point to each other, connect them with an edge and remove one stub
from each.  Erase all arrows and repeat.  After some finite number of such
iterations, no more edges are added. The core multi-matching of $(P,D)$ in
$S$ defined to be the resulting graph. Note that the degree of a vertex $x\in
P$ is at most $D_x$, but may be strictly less.
\begin{lemma}
Let $P$ be any discrete set of points in $\R^d$, let
$(D_x)_{x\in P}$ be positive integers, and let $S$ be a bounded
set.  Every edge in the core multi-matching of $(P\cap S,D)$ in
$S$ is present in every stable multi-matching of $(P,D)$ on
$\R^d$.
\end{lemma}
\begin{proof}
It is straightforward to check by induction on the steps of the
above algorithm that every edge added is present in any stable
multi-matching. The key point is that if $x\in S$ is closer to
some other point in $y\in S$ than to $S^C$, then $x$ is also
closer to $y$ than to any point in $P\setminus S$.
\end{proof}
We now specialize to the main case of interest.  Let $d=1$, and let
$S=I=[a,b]$, a bounded interval.  Let $\poi$ be a Poisson process of
intensity $1$ on $\R$, and consider the case $\mu(\{2\})=1$ of
deterministically two stubs per vertex.  By the core multi-matching on the
interval $I$ we mean the core multi-matching of $([\poi]\cap I,D)$ on $I$,
where $D\equiv 2$. We call an interval $I=[a,b]$ \df{good} if the core
matching on $I$ has a connected component with a point in the first quarter
$[a,\tfrac34a+\tfrac14b]$ and a point in the last quarter
$[\tfrac14a+\tfrac34b,b]$.

\begin{theorem}\label{conditional}
Let $\poi$ be a Poisson process of intensity 1 on $\R$ and let
$\mu(\{2\})=1$.  If for some $L$ we have $\PP([0,L] \text{ is good})>0.968$,
then the 2-stub stable multi-matching has an infinite component.
\end{theorem}

Monte-Carlo simulations provide overwhelming evidence that the
condition in Theorem~\ref{conditional} indeed holds for some large $L$,
subject to the trustworthiness of the pseudo-random number generator and the
software used. Indeed, in $1000$ independent simulation runs of the process
with $L=13000$, the interval $[0,L]$ was good in $991$ cases, implying that
the hypothesis that the probability is $0.968$ or less can be rejected at the
$10^{-6}$ level.  See the appendix for details.

By a \df{monotone path} in a multi-matching we mean a sequence of vertices
$x_1<x_2<\cdots<x_k$ with the edges
$(x_1,x_2),(x_2,x_3),\ldots,(x_{k-1},x_k)$ all present.  As observed in the
proof of Theorem \ref{th:stable} (ii), no two edges cross in the 2-stub
stable multi-matching, and hence the same holds in a core multi-matching. If
$I$ is good, it follows that the core multi-matching on $I$ contains a
monotone path from the first quarter to the last quarter.  We call such a
path a \df{spanning path} of the good interval.

\begin{lemma}\label{combine}
Let $a<b<c<d$ be points of $\poi$, and suppose that the intervals $[a,b]$ and
$[c,d]$ are both longer that $[b,c]$.  If the 2-stub stable multi-matching
has a monotone path $\alpha$ from $a$ to $b$ and a monotone path $\delta$
from $c$ to $d$, then it has a monotone path from $a$ to $d$ which contains
$\alpha$ and $\delta$.
\end{lemma}

\begin{proof}
Suppose on the contrary that there is no monotone path from $a$ to $d$
containing $\alpha$ and $\delta$. First extend the path $\alpha$ to the right
as far as possible within $[b,c]$; that is, let $b'\in[b,c]$ be as large as
possible such that there is a monotone path containing $\alpha$ from $a$ to
$b'$. Similarly extend $\gamma$ as far left as possible to $c'\in[b,c]$. By
our assumption, and since there are no crossing edges, we have $b'<c'$. Note
also that $[a,b']$ and $[c',d]$ are longer than $[b',c']$.  Now $b'$ is
adjacent to its neighbour in the monotone path from $a$, and to exactly one
other vertex $x$. By our assumptions, $x\notin [b',c']$, and therefore
$x\not\in(a,d)$, otherwise we would have crossing edges.  A similar argument
shows that $c'$ has a neighbour outside $(a,d)$. But now $(b',c')$ form an
unstable pair.
\end{proof}

\begin{corr}\label{goodness}
If at least $8$ of the $9$ intervals $[0,x],[x,2x],\ldots,[8x,9x]$ are good,
then so is $[0,9x]$.  Furthermore, under the same assumption, given any
spanning paths, one of each of the good short subintervals, there is a
spanning path of the long interval containing all of them.
\end{corr}

\begin{proof}
Let the configuration outside $I:=[0,9x]$ be arbitrary and consider the
stable multi-matching.  Write $I_k=[(k-1)x,kx]$.  For any sequence of
consecutive good intervals $I_a,I_{a+1},\ldots,I_b$, by Lemma~\ref{combine}
we obtain a monotone path in their union reaching to within distance $x/4$ of
each end. If $I_3,I_4,I_5,I_6,I_7$ are all good, then the resulting path
reaches to within $2x+x/4=9x/4$ of each end of $I$, as required.  On the
other hand, if one of $I_3,I_4,\ldots,I_7$ is bad (but the other $8$
subintervals are good), then we obtain two paths of length greater than
$2x-2(x/4)=3x/2$ on either side of the bad subinterval, with a gap of length
less than $x+2(x/4)=3x/2$ in between, so another application of
Lemma~\ref{combine} provides the required spanning path.
\end{proof}

\begin{proof}[Proof of Theorem \ref{conditional}]
Let
$$
I^k:=\Bigl[-\frac{9^kL}{2},\frac{9^kL}{2}\Bigr],
$$
and let $p_k$ be the probability that $I^k$ is good. Thus $p_0>0.968$, and by
Corollary \ref{goodness}, $p_{k+1}\geq f(p_k)$ where
$$f(p):=p^9+9p^8(1-p).$$
It follows by an elementary computation that $p_k\to 1$ as $k\to\infty$,
and indeed $\sum_k(1-p_k)<\infty$. Hence by the Borel-Cantelli lemma, a.s.\
$I^k$ is good for all sufficiently large $k$.  Moreover, since
$\sum_k(1-f(p_k))<\infty$, for all sufficiently large $k$, the interval $I^k$
can be divided into $9$ equal intervals of which at least $8$ are good.  By
Corollary \ref{goodness} it follows that for some (random) $K$ we may find
monotone paths $\pi_K,\pi_{K+1},\pi_{K+2},\ldots$, each contained in the
next, with $\pi_k$ a spanning path of $I^k$ for each $k$.  Then
$\bigcup_{k\geq K} \pi_k$ is an infinite connected graph in the stable
multi-matching.
\end{proof}

\section{Some counterexamples}\label{further}

Theorem \ref{th:smrd} asserts that the random direction stable
multi-matching has no infinite component when all vertices have
degree 2, and that it has long edges in the sense that
$\E^*[X^{1/2}]=\infty$. Furthermore, it follows from Theorem
\ref{th:stable} that existence of an infinite component in the
stable-multi matching with all degrees equal to 2 is equivalent
to $N<\infty$. This might lead one to suspect that there is a
simple relation between the component structure and the edge
length for $\mu(\{2\})=1$ that holds for any matching scheme.
Below, we give two examples of factor matching schemes that
demonstrate that this is not the case.

\paragraph{Example 1.} Our first example is a matching scheme where all
components are infinite and where also the number of edges
crossing the origin is a.s.\ infinite. Note that, if the origin
is crossed by infinitely many edges, then also $N=\infty$ and
thus, by Lemma \ref{le:NM}, $\E^*[M]=\infty$. Existence of an
infinite component hence does not imply short edges in any of
these respects.

To describe the matching scheme, let each point in $[\poi]$ be
equipped with two stubs. Recall that the stable multi-matching
in the special case where $\mu(\{1\})=1$ is known as the stable
matching. First use one stub per point to form edges according
to the stable matching of the points. Then orient the remaining
stub at each point in the opposite direction (left or right)
from that of the first stub, and connect these directed stubs
according to the procedure used for the random direction stable
matching.  This gives a graph where each point has one edge
connected to the right and one edge connected to the left --
that is, all points are birds in the in the terminology used in
the proof of Theorem \ref{th:smrd} -- which implies that all
components in the graph are infinite. That the number of edges
crossing the origin is infinite almost surely follows from
Proposition \ref{prop:odd_deg} applied to the configuration
after the first stub per point is
connected.\hfill$\Box$\medskip

\paragraph{Example 2.} The next example is a matching scheme that
gives almost surely only finite components and where the
expected edge length is finite. Finite expected edge length
hence does not imply existence of an infinite component.

The matching scheme proceeds by dividing the vertices into
groups of size at least $3$ as follows. Call a point of $\poi$
a {\bf seed} if it has some other point within distance $1$.
Call a seed $x$ {\bf good} if the number of non-seed points
between $x$ and the next seed to its right is at least $2$. Now
whenever $x<y$ are two consecutive good seeds, let all the
points in $[x,y)$ constitute one group.

Define the matching as follows. For a group
$x_1<\cdots<x_k$, connect the two stubs per vertex to form the
edges $(x_i,x_{i+1})$ for $i=1,\ldots,k-1$ and the edge
$(x_1,x_k)$. This clearly gives a configuration with almost
surely finite components and finite expected edge
length.\hfill$\Box$\medskip

Next we give simple examples of translation-invariant point processes on
$\RR$ for which the stable multi-matching in the case $\mu(\{2\})=1$ provably
does, and does not, have an infinite component.

\paragraph{Example 3.} Let $(X_i)_{i\in\Z}$ be i.i.d.\ and uniformly
distributed on $[0,1/3]$, and let $U$ be independent and uniform on $[0,1]$.
Consider the point process with support $\{i+X_i+U:i\in\Z\}$.  It is easy to
see that each point connects to its left-neighbour and its right-neighbour,
so there is an infinite component.

\enlargethispage*{1cm}
\paragraph{Example 4.}  \sloppypar
Let $(X_{i,j})_{i\in\Z,\,j=1,2,3}$ be i.i.d.\ and uniformly distributed on
$[0,1/3]$, and let $U$ be independent and uniform on $[0,1]$. Consider the
point process with support $\{i+X_{i,j}+U:i\in\Z,j=1,2,3\}$. Then each
component has size exactly $3$.

\section{Open problems}\label{open}

\subsubsection*{The random direction stable multi-matching}

For degree distributions other than $\mu(\{2\})=1$, it remains
an open problem to determine if the random direction stable
multi-matching generates an infinite component.

\subsubsection*{The stable multi-matching}

\begin{table}\centering
\resizebox{\textwidth}{!}{%
\begin{tabular}{|c|c||c|c|c|c|c|}
  \hhline{~~|-----}
  \multicolumn{2}{c|}{} & \multicolumn{5}{|c|}{Number of points}\\
  \hhline{~~|-----}
  \multicolumn{2}{c|}{} & $2^{10}$ & $2^{12}$ & $2^{14}$ & $2^{16}$ & $2^{18}$ \\
  \hhline{--|=====}
  \multirow{10}{*}{\rotatebox{90}{Expected degree}}
  & 2 & $.244\pm.099$ & $.291\pm.044$ & $.297\pm.014$ & $.287\pm.009$ & $.292\pm.005$ \\
  \cline{2-7}
  & 3 & $.278\pm.099$ & $.154\pm.029$ & $.049\pm.006$ & $.017\pm.003$ & $.006\pm.001$ \\
  \cline{2-7}
  & 4 & $.802\pm.158$ & $.728\pm.232$ & $.653\pm.201$ & $.366\pm.207$ & $.399\pm.143$ \\
  \cline{2-7}
  & 5 & $.974\pm.018$ & $.933\pm.114$ & $.815\pm.183$ & $.672\pm.258$ & $.321\pm.102$ \\
  \cline{2-7}
  & 6 & $.989\pm.009$ & $.990\pm.004$ & $.975\pm.047$ & $.933\pm.161$ & $.755\pm.207$ \\
  \cline{2-7}
  & 7 & $.994\pm.008$ & $.997\pm.003$ & $.989\pm.020$ & $.996\pm.000$ & $.961\pm.112$ \\
  \hhline{|~|======}
  & 2.1 & $.071\pm.023$ & $.024\pm.005$ & $.009\pm.002$ & $.003\pm.000$ & $.001\pm.000$\\
  \cline{2-7}
  & 2.5 & $.132\pm.062$ & $.043\pm.018$ & $.018\pm.006$ & $.006\pm.001$ & $.001\pm.000$\\
  \cline{2-7}
  & 3.5 & $.472\pm.172$ & $.244\pm.071$ & $.146\pm.061$ & $.050\pm.014$ & $.016\pm.004$\\
  \cline{2-7}
  & 4.5 & $.992\pm.011$ & $.888\pm.110$ & $.529\pm.138$ & $.298\pm.132$ & $.129\pm.059$\\
  \hline
\end{tabular}
}
 \caption{Simulation results for the stable multi-matching of uniformly
random points on the cycle.  The proportion of points in the largest
connected component is indicated as ``sample mean $\pm$ sample standard
deviation'' for a sample of size $10$.  The degree $D$ is either a constant
integer, or takes two consecutive integer values with probabilities
determined by the indicated expected value.}\label{sim}
\end{table}


Firstly, it would of course be desirable to turn the ``statistical proof'' of
percolation for $D\equiv 2$ in $d=1$ into a fully rigorous proof.
Furthermore, it remains an open problem to determine whether there exists an
infinite component the stable multi-matching for other degree distributions
(an exception being the case $D\in\{1,2\}$ with $\PP(D=1)>0$; see
\cite{DHH}). Another interesting case arises when most vertices have $2$
stubs, but a small fraction have $3$; this case can be expected to be very
different from the 2-stub case since there are local configurations which can
end a long path. Indeed, simulations appear to indicate that the proportion of vertices in the
largest component converges to $0$ as the system size increases, thus
suggesting no infinite component; see Table~\ref{sim} (lines 9-10). The observation that
local modifications can end a long path extends to any degree distribution
with support on at least one odd integer. Is it the case that the stable
multi-matching has an infinite component in $d=1$ if and only if the degree
distribution has support only on even integers?  The results in
Table~\ref{sim} appear consistent with this hypothesis.

\subsubsection*{Iterated stable matching}

Yet another multi-matching scheme is obtained by repeatedly
applying the stable matching of the points with the restriction
that multiple edges are not allowed. More specifically, first
take the stable matching of $[\poi]$, connect the points
accordingly and remove one stub per point. Then consider the
stable matching of the points that have at least one stub left
on them and with the modification that two points that already
have an edge between them cannot be matched. This matching is
obtained by repeatedly matching mutually nearest neighbors in
the set of points with at least one stub left on them, avoiding
matchings of points that are already connected. As remarked in
\cite[Remark 2.2]{DHH}, the proof of \cite[Proposition
2.2]{DHH} is easily modified to show that this yields a perfect
matching. Connect the points according to this matching and
remove one stub from each point that is connected. Repeat
indefinitely.

Does this matching scheme generate an infinite component? How
does the answer depend on the degree distribution? Note that it
follows from Proposition \ref{prop:odd_deg} that the number of
edges that cross the origin is a.s.\ infinite already after the
first stub of the vertices has been connected. For degree
distributions with degrees larger than 1 however the matching
contains crossing edges. This means for instance that the proof
of Theorem \ref{th:stable}(b) cannot be applied to draw the
same conclusion (that $N<\infty$ if there is an infinite
component) for the iterated stable matching.

\section*{Appendix -- simulation code}

The Python 2.6 code below was used to verify that the interval $[0,13000]$
was a good interval in 991 out of 1000 cases. Since
$$\PP\bigl[\mbox{Binomial}(1000,0.968)\geq 991\bigr]<10^{-6},
$$
this gives grounds to reject the hypothesis that $\PP([0,13000] \text{ is
good})\leq 0.968$ at the $10^{-6}$ level.

The code uses Python's built-in implementation of the Mersenne Twister, one
of the most extensively tested pseudo-random number generators.  The
experiment was also repeated using two other pseudo-random number generators,
and using an alternative method of generating a Poisson point process, with
consistent results.

The following observations simplify and speed up the construction of the core
multi-matching in the $2$-stub $1$-dimensional case.  In place of the
complement of an interval $[a,b]$, it suffices to consider distances to the
endpoints $\{a,b\}$.  At any stage of the core multi-matching algorithm, the
arrow from a point $x$ with an unused stub must point to one of $x$'s two
closest neighbours on the left or its two closest neighbours on the right
among the set of points of $\widetilde P$ with unused stubs; this is because
at most one other point can have an edge to $x$ already.  Also, if the arrows
of $x,y$ point to each other, then there is no other point $z$ with unused
stubs between them; indeed such a $z$ cannot already have edges to both $x$
and $y$, so one of $x$ and $y$ would instead point to $z$.

{\scriptsize
\begin{verbatim}
from random import *
from math import *

def poi(a):        # Poisson random variable with mean a
    t=-1
    while a>0:
        a+=log(random())
        t+=1
    return t

def setup(a):
    global n,x,stubs,e
    n=poi(a)
    x=[random() for i in xrange(n)]
    x.sort()       # x = sorted list of n=poi(a) random points in [0,1]
    stubs=[2]*n    # 2 stubs per point
    e=set([])      # e = set of edges

def xx(i):         # position of point i, or endpoint of [0,1] for i outside range
    if i<0:
        return 0.
    elif i>=n:
        return 1.
    else:
        return x[i]

def corematch():
    cont=True
    while cont:
        cont=False
        active=[-1,-1]+[i for i in xrange(n) if stubs[i]]+[n,n]
                         # points with unused stubs, plus 2 dummy points at each end
        arrow=[None]*len(active)
        for j in xrange(2,len(active)-2):
            l=[(abs(xx(active[j])-xx(active[k])),k)     # find distances
               for k in j-2,j-1,j+1,j+2                 # to 2 neighbours on each side
               if tuple(sorted([active[j],active[k]])) not in e]  # if no edge already
            if l:
                arrow[j]=min(l)[1]                        # arrow to closest
        for j in xrange(2,len(active)-3):
            if arrow[j]==j+1 and arrow[j+1]==j:         # found pair with mutual arrows
                                                        # (must be neighbours)
                e.add((active[j],active[j+1]))          # add edge and remove stubs
                stubs[active[j]]-=1
                stubs[active[j+1]]-=1
                cont=True                               # keep going if some edge added

def components():                                       # find components of graph
    nbrs=dict((i,[]) for i in xrange(n))
    for (i,j) in e:
        nbrs[i].append(j)
        nbrs[j].append(i)
    done=[False]*n
    ans=[]
    for i in xrange(n):
        if not done[i]:
            cur=[i]
            done[i]=True
            for j in cur:
                for k in nbrs[j]:
                    if not done[k]:
                        cur.append(k)
                        done[k]=True
            ans.append(sorted(cur))
    return ans

def good():                                             # is the interval good?
    return any(x[c[0]]<.25 and x[c[-1]]>.75 for c in components())

seed(12345)
a=13000
k=1000
g=0
for i in xrange(k):
    setup(a)
    corematch()
    if good():
        g+=1
    print g,'/',(i+1),'..',
print
print 'Interval of length',a,'was good',g,'times out of',k
\end{verbatim}
}

\end{document}